\documentclass{article}

\usepackage{amsmath,amssymb,amscd,amsthm}
\usepackage{diagrams}
\input cyracc.def

\newcommand{\QQ}{\mathbb Q}
\renewcommand{\AA}{\mathbb A}
\newcommand{\ZZ}{\mathbb Z}
\newcommand{\NN}{\mathbb N}
\newcommand{\FF}{\mathbb F}

\newcommand{\BB}{\mathbb B}
\newcommand{\CC}{\mathbb C}
\newcommand{\EE}{\mathbb E}
\newcommand{\calC}{{\cal C}}
\newcommand{\Qp}{\QQ_p}
\newcommand{\Zp}{\ZZ_p}
\newcommand{\mup}{\mu_{p^\infty}}
\newcommand{\tsigma}{\tilde{\sigma}}
\newcommand{\EfK}{\EE_K}

\newcommand{\frakT}{{\frak T}}
\newcommand{\calx}{{\frak x}}
\newcommand{\caly}{{\frak y}}
\newcommand{\calz}{{\frak z}}
\newcommand{\tM}{\tilde{M}}
\newcommand{\mun}{\mu_{p^n}}

\newcommand{\vs}{\vspace{1ex}}
\newcommand{\xin}{\xi_{p^n}}
\newcommand{\fT}{\frak{T}}
\newcommand{\gammaone}{\gamma_{1,n}}
\newcommand{\gammatwo}{\gamma_{2,n}}
\newcommand{\gammatwoN}{\gamma_2^{\frac{1}{N}}}

\newcommand{\bgamma}{b_{\gammatwo}}
\newcommand{\tbgamma}{\tilde{b}_{\gammatwo}}
\newcommand{\agamma}{a_{\gammaone}}
\newcommand{\tagamma}{\tilde{a}_{\gammaone}}
\newcommand{\calA}{{\cal A}}
\newcommand{\fm}{\frak{m}}
\newcommand{\bpi}{\bar{\pi}}
\newcommand{\calF}{{\cal F}}
\newcommand{\calG}{{\cal G}}
\newcommand{\calR}{\tilde{\EE}}
\newcommand{\frakV}{\frak{V}}

\DeclareMathOperator{\Gal}{Gal}

\DeclareMathOperator{\id}{id}

\DeclareMathOperator{\Tr}{Tr}
\DeclareMathOperator{\Res}{Res}
\DeclareMathOperator{\TR}{TR}
\DeclareMathOperator{\Hom}{Hom}

\DeclareMathOperator{\cris}{cris}
\DeclareMathOperator{\dR}{dR}
\DeclareMathOperator{\Fil}{Fil}

\newtheorem{thm}{Theorem}[section]
\newtheorem{prop}[thm]{Proposition}
\newtheorem{lem}[thm]{Lemma}
\newtheorem{cor}[thm]{Corollary}

\begin{document}

\title{The higher Hilbert pairing via $(\phi,G)$-modules}

\author{Sarah Livia Zerbes}

\maketitle


\begin{abstract}
 Following the strategy in~\cite{herr2}, we prove a Tate duality for higher dimensional local fields of mixed
 characteristic $(0,p)$, $p\neq 2$. The main tool is the theory of higher fields of norms as developed
 in~\cite{andreatta} and~\cite{scholl}. Assuming that $p$ is not ramified in the basefield, 
 we then use this construction to define the higher Hilbert
 pairing. In particular, we show that the Hilbert pairing is non-degenerate, and we re-discover the
 formulae of Br\"uckner and Vostokov.

\end{abstract}

\tableofcontents


\section{Introduction}


 \subsection{Statement of the main result}
 
  Let $p$ be an odd prime, and let $K$ be a $d$-dimensional local field of mixed characteristic $(0,p)$
  Denote by $\calG_K$ the absolute Galois group $\Gal(\bar{K}\slash K)$. For $n\geq 1$, denote by
  $\mun$ the group of $p^n$th roots of unity. 
  \vs
  
  This paper consists of two parts. In the first part, we prove a higher Tate duality for the
  $\calG_K$-module $\mun$:-
  
  \begin{thm}\label{Duality}
   Let $K$ be a $d$-dimensional local field of mixed characteristic $(0,p)$, and denote by $\calG_K$ its
   absolute Galois group. Let $F$ be the maximal algebraic extension of $\QQ_p$ contained in $K$, and assume that
   $O_K\slash O_F$ is formally smooth. 
   Then for $i\in\{0,\dots,d+1\}$ and all $n\in\NN$ we have a perfect pairing 
   \begin{equation*}
    H^i(\calG_K,\mun^{\otimes i})\times H^{d+1-i}(\calG_K,\mun^{\otimes d-i})\rightarrow \QQ_p\slash\ZZ_p.
   \end{equation*}
  \end{thm}

  \noindent {\it Remark.} The above pairing should certainly be the same as the cup product pairings, but
  this is not that easy to show.
  \vs
  
  \noindent To prove Theorem~\ref{Duality},  we follow the strategy of Herr in~\cite{herr2} and 
  express the Galois cohomology groups in terms of the $(\phi,G)$-module of $\mun$. 
  We will also prove a higher dimensional Tate isomorphim:-

  \begin{thm}\label{Tateisom}
   Let $K$ be a $d$-dimensional local field of mixed characteristic $(0,p)$. 
   Let $F$ be the maximal algebraic extension of $\QQ_p$ contained in $K$, and assume that
   $O_K\slash O_F$ is formally smooth. Then there is a canonical isomorphism
   \begin{equation*}
    H^{d+1}(\calG_K,\mup^{\otimes d})\cong \Qp\slash\Zp.
   \end{equation*}
  \end{thm}
  
  In the second part of the paper, we assume that $p$ is prime in $K$, and 
  we use Theorem~\ref{Duality} to define a pairing
  \begin{equation*}
   K_d(K_n)\times K_1(K_n)\rightarrow \mun.
  \end{equation*}
  Composing it with the natural multiplication map $K_1(K_n)^{\times d}\rightarrow K_d(K_n)$, we obtain a
  pairing $V_n:K_1(K_n)^{\times (d+1)}\rightarrow \mun$
  which factors through
  \begin{equation*}
   \frakV_n:\big(K_1(K_n)\slash p^n\big)^{\times (d+1)}\rightarrow \mun.
  \end{equation*}  
  In Section~\ref{higherhilbert} we give a an explicit description of $\frakV_n$:- 
  For $1\leq i\leq d+1$, let $\alpha_i\in O_K^\times$ such that $\alpha_i\cong 1\mod
  \bpi_n$,  and let $F_i(X)\in \AA_K^+$ such
  that $h_n(F_i)=\alpha_i$. Let $f_i(X)=(1-\frac{\phi}{p})\log F(X)$.
  
  \begin{thm}\label{formulaehilbert}
   The pairing $\frakV_n$ is non-degenerate. Moreover, we have 
   \begin{equation*}
    \frakV_n(\alpha_1,\dots,\alpha_{d+1})=\mun^{\Tr \Res_{\pi_n,T_1,\dots,T_d}(\Phi)},
   \end{equation*}
   where $\Phi$ is given by the formula
   \begin{align*}
    \Phi=\frac{1}{\pi} \sum_{i=1}^{d+1}\frac{(-1)^{d+1-i}}{p^{d+1-i}}f_i(\pi_n)d\log
    F_1(\pi_n)&\wedge\dots\wedge d\log F_{i-1}(\pi_n)\\
    &\wedge d\log F_{i+1}^\phi(\pi_n)\wedge \dots\wedge
    d\log F_{d+1}^\phi(\pi_n).
   \end{align*}
  \end{thm}
  
  Comparing these formulae with the explicit descriptions of the higher Hilber pairing of 
  Br\"uckner and Vostokov (c.f.~\cite{brueckner} and~\cite{vostokov}), we get the following result:-
  
  \begin{cor}
   The pairing $\frakV_n$ is the higher Hilbert pairing.
  \end{cor}
  
  For proving Theorem~\ref{formulaehilbert}, we follow the strategy of Benois in~\cite{benois1}. 
  \vs
  
  \noindent {\bf Remarks.} (1) To keep the notation as simple as possible, we will prove the above results for 
  local fields of dimension
  $2$. However, the proofs generalize without problems to local fields of arbitrarily high dimension.  
 
  \noindent (2) Theorem~\ref{Duality} can certainly be generalized to an arbitrary 
  $\ZZ_p$- of $\calG_K$. We will deal with this in a different paper. 
 
 

 \subsection{Notation}
 
  \noindent $*$ For a $2$-dimensional local field $K$ with ring of integers $O_K$, let 
  $k_K\cong \FF_p((T))$ be the residue field.
  
  \noindent $*$ For a $(\phi,G)$-module $M$, we sometimes 
  denote the action of the Frobenius operator on $M$ by $\phi_M$. 
  
  \noindent $*$ For a $2$-dimensional local field $K$, let $\calG_K=\Gal(\bar{K}\slash K)$. 

  
 \subsection{Acknowledgements}
 
  I am very grateful to John Coates for his interest and encouragement. 
  Also, I would like to thank Otmar Venjakob, Kay Wingberg and Ivan Fesenko for some helpful comments.
  Finally, I would like to very warmly thank Guido Kings for his invitation to Regensburg in Spring 2007
  when part of this paper was written. 


\section{Higher $(\phi,G)$-modules}


 \subsection{Setup}
 
  Let $K$ be a $2$-dimensional local field of mixed characteristic $(0,p)$, and let $F$ be the maximal
  algebraic extension of $\QQ_p$ contained in $K$. Let $k_F$ be the residue field of $F$
  and let $\omega_F$ be a uniformizer of $F$. Assume that $O_K\slash O_F$ is formally smooth, i.e.
  $\omega_F$ is a uniformizer of $K$. Let $X$ be a 
  unit in $K$ whose reduction $\bar{X}$ is a $p$-basis for the residue field $k_K$ of
  $K$, so $k_K\cong k_b((X))$ for some finite extension $k_b$ of $k_F$. 
  Fix an algebraic closure $\bar{K}$ of $K$. Let $(\xi_i)_{i\geq 0}$ be  compatible system of primitive 
  $p^i$th roots of unity, and  let $(X_{i})_{i\geq 0}$ be a compatible system of
  $p^i$th roots of $X$. Denote by $\mu_{p^i}$ the group of $p^i$th roots of unity. 
  \vs
  
  Let $K_i=K(\mu_{p^i},X_{i})$ and $K_\infty=\bigcup K_i$. Also, let $F_i=F(\mu_{p^i})$.
  
  \begin{lem}\label{structureGaloisgroup}
   The extension $K_\infty$ is a $2$-dimensional $p$-adic Lie extension of $K$. More precisely, we have 
   $\Gal(K_\infty\slash K)\cong \Gamma_1\rtimes \Gamma_2$, where
   \begin{equation*}
    \Gamma_2=\Gal(K_\infty\slash K(\mu_{p^\infty}))\cong\ZZ_p.
   \end{equation*}
   and $\Gamma=\Gal(K(\mu_{p^\infty})\slash K)$ is isomorphic (via the cyclotomic character $\chi$) 
   to an open subgroup of $\ZZ_p^\times$.
  \end{lem}
  
  Let $\gamma_1$ and $\gamma_2$ be topological generators of $\Gamma_1$ and $\Gamma_2$, respectively. 
  Let $a\in\ZZ_p$ such that
  \begin{equation}\label{conjugation}
   \gamma_1\gamma_2=\gamma_2^a\gamma_1.
  \end{equation}
  
  \noindent {\bf Note.} We have $a=\chi(\gamma_1)\in
  \ZZ_p^\times$. It follows that in particular we have 
  \begin{equation}\label{reverseconjugation}
   \gamma_1\gamma_2^{\frac{1}{a}}=\gamma_2\gamma_1.
  \end{equation}

  Let $\EE_F$ be the field of norms of the tower $(F_i)_{i\geq 0}$, and let $k_{\calF}$ be its 
  residue field. Let $\bpi_F$ be a uniformizer of $\EE_F$, so $\EE_F\cong k_{\calF}((\bpi_F))$.
  Let $\EE_K$ be the field of norms of the tower $(K_i)$. Let $\epsilon=(1,\xi_1,\xi_2,\dots)$ and 
  $\frakT=(X_{i})_{i\geq 1}\in\EE_K.$ Define $\bpi=\epsilon-1$. Let $k=k_bk_{\calF}$. 
  
  \begin{lem}
   The field $\EE_K$ is given by
   \begin{align*}
     \EfK &\cong k_{\calF}((\bpi_F))\hat{\otimes}_{k_F}k_b((\frakT))\\
          &\cong k((\frakT))((\bar{\pi}_F)).
   \end{align*}
  \end{lem}
  \begin{proof}
   See the section on Kummer towers in~\cite{scholl}.
  \end{proof}


 \subsection{Lift to characteristic $0$}\label{liftchar0}
    
  Let $\AA_F$ be a lift of $\EE_F$ to characteristic $0$, so $\AA_F\cong W(k_{\calF})[[\pi_F]]
  [\pi_F^{-1}]^\vee$, where $\pi_F$ is a lift of $\bpi_F$. 
  Let $\phi$ be a lift to $\AA_F$ of the Frobenius
  operator  commuting with the action of $\Gamma_1$. Let $T=[\frakT]$. Define
  \begin{equation*}
   \AA_K=W(k)[[T]][T^{-1}]^\vee[[\pi_F]][\pi_F^{-1}]^\vee.
  \end{equation*}
  Then $\AA_K$ is a lift of $\EE_K$ to characteristic $0$. Let $\BB_K=\AA_K[\frac{1}{p}]$ be
  its field of fractions. Note that $\AA_F\subset\AA_K$. Define a lift of Frobenius to
  $\AA_K$ by $\phi(T)=T^p.$
  Note that $\phi$ commutes with the action of $G$ on $\AA_K$.  
  Define $N\in\ZZ_p$ by
  \begin{align*}
   \gamma_2(\pi_F)&=\pi_F,\\
   \gamma_2(T)&=(\pi+1)^{N}T.
  \end{align*}
  One can show that $N\in\ZZ_p^\times$ since $X$ is a $p$-basis of $k_K$. 
  \vs  
  
  Note that $\AA_K$ is a free finitely generated module over $\phi(\AA_K)$ of 
  degree $p^2$. It follows that
  we can define a left inverse $\psi$ of $\phi$ by the formula
  \begin{equation*}
   \phi(\psi(x))=\frac{1}{p^2}\Tr_{\AA_K\slash\phi(\AA_K)}(x).
  \end{equation*}
  \vs
  
  \begin{prop}
   Let $\EE$ be an algebraic closure of $\EfK$. Then we have an isomorphism of Galois groups
   \begin{equation*}
    \Gal(\EE\slash\EfK)\cong \Gal(\bar{K}\slash K_\infty).
   \end{equation*}
  \end{prop}
  \begin{proof}
   See~\cite{scholl}.
  \end{proof}
  
  Let $\AA$ be a lift of $\EE$ containing $\AA_K$. Then the actions of $\phi$ and $\psi$ can be extended 
  uniquely to $\AA$ (c.f.~\cite{scholl}).
   

 \subsection{The ring $\AA_{\cris}$}
 
   Let $\CC_K$ be the $p$-adic completion of $\bar{K}$, and let $O_{\CC_K}$ be its ring of integers. Let
   $\calR$ be the set of sequence $x=(x^{(0)},x^{(1)},\dots)$ of elements in $O_{\CC_K}$ satisfying
   $(x^{(i+1)})^p=x^{(i)}$. Then $\calR$ has a natural structure as a ring of charactristic $p$. 
   For $n\geq 1$,
   let $\epsilon_n=(\zeta^{(n)},\zeta^{(n+1)},\dots)$ be the $p^n$the root of $\epsilon$ in $\calR$. Let
   $\AA_{\inf}=W(\calR)$ be the ring of Witt vectors of $\calR$, $\phi$ the Froenius of $\AA_{\inf}$, and
   if $x\in\calR$, then denote by $[x]$ its Teichm\"uller representative in $\AA_{\inf}$. Then the
   homomorphism
   \begin{align*}
    \theta: &\AA_{\inf}\rightarrow O_{\CC_K}\\
            &\sum p^n[x_n]\rightarrow \sum p^n x_n^{(0)}
   \end{align*}
   is surjective and its kernel is a principal ideal with generator
   $\omega=\frac{[\epsilon]-1}{[\epsilon_1]-1}$. Let $\BB_{\inf}=\AA_{\inf}(p^{-1})$. Note that $\theta$
   extends to a homomorphism $\BB_{\inf}\rightarrow \CC_K$. Define $\BB_{\dR}^{\nabla +}=\varprojlim
   \BB^+_{\inf}\slash (\ker \theta)^n$, and extend $\theta$ by continutiy to a homomorphism 
   $\BB_{\dR}^{\nabla +}\rightarrow \CC_p$. This makes $\BB_{\dR}^{\nabla +}$ into a discrete valuation
   ring with maximal ideal $\ker(\theta)$ and residue field $\CC_K$. The action of $\calG_K$ on
   $\BB^+_{\inf}$ extends by continuity to a continuous action on $\BB^{\nabla}+_{\dR}$.
   \vs
   
   Let $\AA_{\cris}$ be the subring of $(\BB_{\dR}^\nabla)^+$ consisting of the elements of the form
   $\sum_{n=0}^{\infty}a_n\frac{\omega^n}{p!}$, where $a_n$ is a sequence of elements in $\AA_{\inf}$
   tending to $0$ as $n\rightarrow +\infty$.


 \subsection{Differentials, residues and duality}\label{residuesandduality}
 
  Let $\Omega^1_{\AA_K}$ be the module of continuous $\ZZ_p$-linear $1$-differentials of $\AA_K$.
  Note that we have an isomorphism of $\AA_K$-modules
  \begin{equation*}
   \Omega^1_{\AA_K}\cong \AA_K d\pi_F\oplus \AA_K dT.
  \end{equation*}

  Let $\Omega^1_{\AA_K}$ be the module of continuous $\ZZ_p$-linear $1$-differentials of $\AA_K$, and let
  \begin{equation*}
   \Omega^2_{\AA_K}=\bigwedge^2 \Omega^1_{\AA_K}.
  \end{equation*}

  \begin{lem}\label{structuredifferentials}
   We have an isomorphism of $\AA_K$-modules
   \begin{equation*}
    \Omega^2_{\AA_K}\cong \AA_K d\pi_F\wedge dT.
   \end{equation*}
  \end{lem}
  
  \begin{cor}\label{basechange}
   If $K'$ is a finite separable extension of $K$, then the natural map
   \begin{equation*}
    \AA_{K'}\otimes_{\AA_K}\Omega^2_{\AA_K}\rightarrow \Omega^2_{\AA_{K'}}
   \end{equation*}
   is an isomorphism.
  \end{cor}
  \begin{proof}
   Clear.
  \end{proof}
  
  \begin{cor}\label{unrambasechange}
   If $K'=K\otimes_{F}F'$ for some finite unramified extension $F'$ of $F$, then the natural map
   \begin{equation*}
    W(k_{F'})\otimes_{W(k_F)}\Omega^2_{\AA_K}\rightarrow \Omega^2_{\AA_{K'}}
   \end{equation*}
   is a $W(k_F)$-linear isomorphism.
  \end{cor}
  
  It follows from Lemma~\ref{structuredifferentials} that if $\omega\in\Omega^2_{\AA_K}$, then there 
  exist $a_{i,j}\in\ZZ_p $ such that $\omega=(\sum a_{i,j}T^i\pi_F^j)d\pi_F\wedge dT$.
  \vs
  
  \noindent {\bf Definition.} Define the residue map
  \begin{align*}
   \Res: \Omega^2_{\AA_K}&\rightarrow \ZZ_p,\\
   \Res(\omega)&=a_{-1,-1}.
  \end{align*}
  
  Since $\phi$ is a lift of the Frobenius operator, we have $\phi(\pi_F)=u\pi_F^p$ for some
  $u\in\AA_K^\times$ satisfying $u\cong 1\mod p$. 
  
  \begin{lem}
   For any $\lambda\in\AA_K$, we have
   \begin{equation*}
    \Res(\phi(\lambda) d\phi(\pi_F)\wedge d\phi(T))=p^2\phi(\Res(\lambda d\pi_F\wedge dT)).
   \end{equation*}
  \end{lem}
  \begin{proof}
   It is sufficient to prove the formula for $\lambda=\frac{1}{\pi_FT}$. Write $u=1+pa$ for some
   $a\in\AA_K$. We have $\phi(T)=T^p$, so 
   \begin{equation*}
    \phi(\lambda) d\phi(\pi_F)\wedge d\phi(T)=\frac{p^2}{\pi_FT}d\pi_F\wedge dT+\frac{p}{uT}du\wedge dT.
   \end{equation*}
   But $\frac{1}{uT}du\wedge dT=T^{-1}\sum d(\frac{(-1)^{r+1}p^r}{r}a^r)\wedge dT$. It is easy to see that
   the coefficient of $\pi_F^{-1}$ in $\sum d(\frac{(-1)^{r+1}p^r}{r}a^r)$ is $0$, which
   finishes the proof. 
  \end{proof}
  
  \noindent {\bf Definition.} Let $M$ be a torsion $(\phi,G)$-module. Define
  \begin{equation*}
   \tM=\Hom_{\AA_K}(M,\BB_K\slash\AA_K\otimes_{\AA_K}\Omega^2_{\AA_K}).
  \end{equation*}
  
  \begin{lem}\label{residueisom}
   The residue map induces an isomorphism
   \begin{equation*}
    \TR: \tM\rightarrow M^\vee.
   \end{equation*}
  \end{lem}
  \begin{proof}
   We imitate the proof of Lemma 1.3 in~\cite{herr2}. By continuity, the residue gives a homomorphism from
   $\tM$ to $M^\vee$. Now $\Omega^2_{\AA_K}$ is a free $\AA_K$-module of rank $1$. The ring $\AA_K$ is
   principal and the $\AA_K$-module $\BB_K\slash\AA_K\otimes_{\AA_K}\Omega^2_{\AA_K}$ is divisible and
   hence injective, and the functor which associates $\tM$ to $M$ is exact. Also, the functor which to $M$
   associates $M^\vee$ is exact, so by the snake lemma we can assume without loss of generality that $M$
   is an $\AA_K$-module of length $1$ and hence a $1$-dimensional vector space over $\EE_K$. By choosing a
   basis, we can therefore assume that $M=\EE_K$. Note that we can identify $\EE_K^\vee$ with
   $\Hom_{\ZZ_p}(\EE_K,\FF_p)$ and $\widetilde{\EE_K}$ with $\Hom_{\EE_K}(\EE_K,\Omega^2_{\AA_K}
   \slash p\Omega^2_{\AA_K})$.
   We need to show that the natural map
   \begin{equation*}
    \widetilde{\EE_K}=\Hom_{\EE_K}(\EE_K,\Omega^2_{\AA_K}\slash p\Omega^2_{\AA_K})\rightarrow
    \Hom_{\ZZ_p}(\EE_K,\FF_p)
   \end{equation*}
   induced by the residue is bijective. Now if $f\in\widetilde{\EE_K}$ is non-zero, 
   then since the image of
   $f$ is an $\EE_K$-vector space of dimension $1$ and the residue map is surjective, it is clear that
   $\TR(f)$ is non-zero. Now let $a\in\Hom_{\ZZ_p}(\EE_K,\FF_p)$. For all $m,n\in\ZZ$, let
   $\alpha_{m,n}=a(T^m\bpi_F^n)$. Since $a$ is continuous, there exists $N\in\NN$ such that 
   $\alpha_{m,n}=0$ for all
   $n\geq N$ and all $m\in\ZZ$. Define 
   \begin{equation*}
    \omega=(\sum_{n\geq -N}\sum_{m\in\ZZ}\alpha_{-m,-n}T^{m-1}\pi_F^{n-1})dT\wedge d\pi.
   \end{equation*}
   Then for all $m,n\in\ZZ$, the class $\mod p$ of $\Res(T^m\pi^n\omega)$ is $\alpha_{m,n}$. It follows
   that if we define $f\in\widetilde{\EE_K}$ by $f(1)=\omega\mod p$, then $\TR(f)=a$.
  \end{proof}


 \subsection{The equivalence of categories}\label{eqcat}
 
  Denote by $\calG_K$ the absolute Galois group $\Gal(\bar{K}\slash K)$. 
  \vs
  
  \noindent {\bf Definition.} A $\ZZ_p$-representation of $\calG_K$ is a $\ZZ_p$-module $V$ of finite type
  equipped with a continuous linear action of $\calG_K$. If $V$ is annihilated by a power of $p$, then it is
  called a $p$-torsion module.
  \vs
  
  \begin{thm}\label{eqofcat}
   The functor $V\rightarrow D(V)=(V\otimes_{\ZZ_p}\AA)^{H_K}$ gives an equivalence of categories 
   \begin{equation*}
    (\text{$\ZZ_p$-representations of $\calG_K$})\rightarrow (\text{\'etale $(\phi,G)$-modules over $\AA_K$}),
   \end{equation*} 
   and an essential inverse is given by 
   $D\rightarrow (\AA\otimes_{\AA_K}D)^{\phi=1}$.
  \end{thm}
  \begin{proof}
   See~\cite{andreatta} or~\cite{scholl}.
  \end{proof}


 \subsection{The module $\ZZ_p(2)$}
  
  For an element $a\in\EE_K$, denote by $[a]$ its Teichm\"uller representative.
  Recall that if $K$ contains a primitive $p$th root of unity, then $[\epsilon]=\pi+1$ is an element of
  $\AA_K$. To simplify notation, let $\Omega(K)=\Omega^2_{\AA_K}$. 
  
  \begin{lem}\label{twistisom}
   Let $K'=K(\mu_p)$, and define an $\AA$-linear map
   \begin{align*}
    \rho_{K'}:\AA\otimes_{\ZZ_p}\ZZ_p(1)&\rightarrow \AA\otimes_{\AA_{K'}}\Omega(K'),\\
    \lambda\otimes\epsilon^2&\rightarrow \lambda\otimes d\log [\epsilon]\wedge d\log T.
   \end{align*}
   Then $\rho_{K'}$ is an isomorphism of $\AA$-modules which does not depend on the choice of the generator
   $\epsilon$ of $\ZZ_p(1)$.
  \end{lem}
  \begin{proof}
   Imitate the proof of Lemma 3.6 in~\cite{herr2}.
  \end{proof}

  Note that we can give $\Omega(K)$ the structure of a $(\phi,G)$-module by defining
  \begin{align*}
   \phi_{\Omega(K)}(x dy\wedge dz)&= \frac{1}{p^2}\phi(x) d\phi(y)\wedge \phi(z),\\
      g(x dy\wedge dz)&= \chi(g)g(x) dg(y)\wedge dg(z)
  \end{align*}
  for $g\in G$. 
  
  \begin{prop}\label{isomtotwist}
   With this stucture as a $(\phi,G)$-module, $\Omega(K)$ is isomorphic to $D(\ZZ_p(2))$.
  \end{prop}
  \begin{proof}
   Let $K'=K(\mu_p)$. By Corollary~\ref{basechange} the natural map $\Omega^2_{\AA_K}\rightarrow
   \Omega^2_{\AA_{K'}}$ is injective. Composing it with the natural $\Gal(\bar{K}\slash K)$-equivariant injection
   $\Omega^2_{\AA_{K'}}\rightarrow \AA\otimes_{\AA_{K'}}\Omega^2_{\AA_{K'}}$ gives an $\AA_K$-linear
   $\calG_K$-equivariant map $\Omega^2_{\AA_K}\rightarrow (\AA_K\otimes_{\AA_{K'}}\Omega^2_{\AA_{K'}})^{H_K}$.
   Explicit calculation shows that this map is also surjective. Composing it with the restriction of 
   $\rho_{K'}^{-1}$ to the points fixed by $H_K$ gives an isomorphism of $\AA_K$-modules between
   $\Omega^2_{\AA_K}$ and $D_K(\ZZ_p(2))$ which commutes with the action of $\calG_K$. It is easy to see that
   it also commutes with the action of $\phi$, which finishes the proof. 
  \end{proof}


 \subsection{Calculation of the Galois cohomology}\label{galcohom}
 
  Let $V$ be a $\ZZ_p$-representation of $\calG_K$, and denote by $D$ the corresponding $(\phi,G)$-module
  $D_K(V)$. Define the following complex:-
  
  \begin{equation}\label{complexphi}
   \calC_{\phi,\gamma_1,\gamma_2}(D): 0  \rTo D \rTo^{f_1} D^{\oplus 3} \rTo^{f_2}   
   D^{\oplus 3}\rTo^{f_3} D \rTo 0, 
  \end{equation}
  where the maps $f_i$ are defined as follows:-
  \begin{align*}
   f_1: x      \rightarrow & [ (\phi-1)x,(\gamma_1-1)x,(\gamma_2-1)x],\\
   f_2: (x,y,z)\rightarrow & [ (\phi-1)y-(\gamma_1-1)x,\\
                           & (\phi-1)z-(\gamma_2-1)x,\\
			   & (\gamma_2-1)y-(\gamma_1\frac{\gamma_2^{\frac{1}{a}}-1}{\gamma_2-1}-1)z ],\\
   f_3: (x,y,z)\rightarrow & (\gamma_2-1)x-(\gamma_1\frac{\gamma_2^{\frac{1}{a}}-1}{\gamma_2-1}-1)y
   -(\phi-1)z.
  \end{align*}
  
  Similarly, define the complex
  \begin{equation}\label{complexpsiprime}
   \calC_{\psi,\gamma_1,\gamma_2}(D): 0  \rTo D \rTo^{g_1} D^{\oplus 3} \rTo^{g_2}   
   D^{\oplus 3}\rTo^{g_3} D \rTo 0, 
  \end{equation}
  where the maps $g_i$ are defined as follows:-
  \begin{align*}
   g_1: x      \rightarrow & [ (\psi-1)x,(\gamma_1-1)x,
   (\gamma_2-1)x],\\
   g_2: (x,y,z)\rightarrow & [ (\psi-1)y-(\gamma_1-1)x,\\
                           & (\psi-1)z-(\gamma_2-1)x,\\
			   & (\gamma_2-1)y-(\gamma_1\frac{\gamma_2^{\frac{1}{a}}-1}{\gamma_2-1}-1)z ],\\
   g_3: (x,y,z)\rightarrow & (\gamma_2-1)x-(\gamma_1\frac{\gamma_2^{\frac{1}{a}}-1}{\gamma_2-1}-1)y-(\psi-1)z.
  \end{align*}

  \noindent {\bf Definition.} Denote by 
  $H^i_{\phi,\gamma_1,\gamma_2}(D)$ (resp. $H^i_{\psi,\gamma_1,\gamma_2}(D))$ the cohomology groups 
  of the complex $\calC_{\phi,\gamma_1,\gamma_2}(D)$ (resp. $\calC_{\psi,\gamma_1,\gamma_2}(D)$).
  
  \begin{prop}\label{Galcohomology}
   Let $V=\ZZ_p(1)$ or $\mun$, and let $D$ be its $(\phi,G)$-module.
   Then for all $0\leq i\leq 3$, we have isomorphisms
   \begin{equation*}
    H^i(\calG_K,V)\cong H^i_{\phi,\gamma_1,\gamma_2}(D)\cong H^i_{\psi,\gamma_1,\gamma_2}(D).
   \end{equation*}
  \end{prop}
  
  \noindent {\it Remark.} This result can certainly also be shown to be true for a general $\ZZ_p$-representation $V$ of
  $\calG_K$. However, in the general case the argument is technically much more complicated, and since
  our main interest is the construction of the Hilbert pairing, we restrict ourselves to the case above.
  We will prove the general case in~\cite{zerbesexp}.
  
  \begin{proof}
   Scholl~\cite{scholl} and Andreatta~\cite{andreatta} have shown that we have isomorphisms
   $H^i(\calG_K,V)\cong H^i_{\phi,\gamma_1,\gamma_2}(D)$. It is therefore sufficient to explicitely give
   isomorphisms $\iota_i:H^i_{\psi,\gamma_1,\gamma_2}(D)\cong H^i_{\phi,\gamma_1,\gamma_2}(D)$.
   
   \item $i=3$:- Let $x\in D$. Since $\psi$ is surjective, we can choose $u\in D$ such that $\psi(u)=x$.
   Define $\iota_3(x)=u$. This is well-defined:- If $u'$ also satisfies $\psi(u')=x$, then $a=u-u'\in
   D^{\psi=0}$. Write $u=\sum_{i\in\ZZ}f_i(\pi_F)T^i$. Then in particular $\psi(f_0(\pi))=0$.
   In Proposition I.5.1 in~\cite{cherbonniercolmez2}, it is shown that $\gamma_1-1$ 
   is invertible on $D_F(V)^{\psi=0}$. Let $h_0=(\gamma_1-1)^{-1}(f_0)$. For $j\neq 0$, let   
   $\alpha_j=\frac{(\gamma_2-1)T^j}{T^j}\in\AA_F^*$. Let $v=\sum_{j\neq 0}\alpha_jf_j(\pi)T^j$.
   Then $u=(\gamma_1-1)h_0+(\gamma_2-1)v$ and hence is zero in 
   $H^3_{\phi_D,\gamma_1,\gamma_2}(D)$.
   
   \item $i=2$:- Let $[x,y,z]\in D^{\oplus 3}$ satisfy $(\gamma_2-1)x-(\gamma_1-1)y-(\psi-1)z=0$. 
   Denote $\iota_2([x,y,z])$ by $[u,v,w]$. 
   Write $(\phi-1)w=\sum_{i\in\ZZ}a_if_i(\pi)T^i$. Note that
   $f_0(\pi_F)\in D_f(\ZZ_p(1))^{\psi=0}$, so by Proposition I.5.1 
   in~\cite{cherbonniercolmez2}, there exists $h_0\in
   D_F(\ZZ_p(1))$ such that $(\gamma_1-1)h_0
   =f_0$. Also, when $i\neq 0$, then $\alpha_i=\frac{(\gamma_2-1)T^i}{T^i}$ is invertible in $\AA_F$. Define
   $u=x+h_0(\pi)$, $v=y+\sum_{j\neq 0}f_j(\pi)\alpha_j^{-1}T^j$ and $w=z$.
   
   \item $i=1$:- Let $y,z\in D$ satisfy
   \begin{equation}\label{H1condition}  
    (\gamma_2-1)y=(\gamma_1\frac{\gamma_2^{\frac{1}{a}}-1}{\gamma_2-1}-1)z
   \end{equation} 
   and $(\psi-1)y=(\psi-1)z=0$. 
   We need to show that there exists $x\in D$ such that
   $(\phi-1)y=(\gamma_1-1)x$ and $(\phi-1)z=(\gamma_2-1)x$. How do we construct this $x$? Write
   $(\phi-1)z=\sum_{i\in\ZZ} f_i(\pi_F)T^i$. By Proposition I.5.1 
   in~\cite{cherbonniercolmez2}, $(\gamma_1-1)^{-1}f_0$ is
   well-defined. When $i\neq 0$, then $\alpha_i=\frac{(\gamma_2-1)T^i}{T^i}\in\AA_F^\times$. Define
   \begin{equation*}
    x=(\gamma_1-1)^{-1}f_0(\pi_F)+\sum_{i\in\ZZ}\alpha_i^{-1}f_i(\pi_F)T^i.
   \end{equation*}
   Using~\eqref{H1condition} it is not difficult to see that $x$ has the required properties. 
   
   \item $i=0$:- $\iota_0=\id$. When $x\in D$ satisfies $\psi x=x$, 
   $\gamma_1\frac{\gamma_2^{\frac{1}{n}}-1}{\gamma_2-1}x=x$ and $\gamma_2x=x$, then it is easy to see
   (using again the result from~\cite{cherbonniercolmez2} mentioned above) that $\phi x=x$. 
   \vs
   
   It is not difficult to see that the above maps are indeed isomorphisms.
  \end{proof}
  
  \noindent {\bf Remark.} If $\gamma_1'$ and $\gamma_2'$ is a different pair of topological generators of
  $\Gamma_1$ and $\Gamma_2$, then the complexes $\calC_{\phi,\gamma_1,\gamma_2}$ and 
  $\calC_{\phi,\gamma_1',\gamma_2'}$ are quasi-isomorphic. More generally, we can replace $\gamma_1$ and
  $\gamma_2$ by $\omega_1\gamma_1$ and $\omega_2\gamma_2$ for any $\omega_1,\omega_2\in\Lambda(G)^\times$.

 
 \subsection{Construction of the pairing}

  \noindent {\bf Definition.} The Pontryagin dual $M^\vee$ of an \'etale $(\phi,G)$-torsion module $M$ 
  is defined as the continuous
  homomorphisms
  \begin{equation*}
   M^\vee=\Hom_{\ZZ_p}(M,\QQ_p\slash\ZZ_p).
  \end{equation*}
  
  As shown in Lemma~\ref{residueisom}, the map $\TR$ induces an isomorphism 
  $\tilde{M}\rightarrow M^\vee$. We can therefore give $M^\vee$ the structure of a $(\phi,G)$-module.
  Denote the operation of Frobenius on it by $\phi_{M^\vee}$. We quote the following result
  from~\cite{herr2}:-

  \begin{prop}\label{homcohomduality}
   Let $\calC=(M^i,d^i:M^i\rightarrow M^{i+1})$ be a cochain comlex of abelian groups which are compact and
   locally separated (with $M^i$ in degree $i$). Suppose that the $d^i$ are strict homomorphisms with
   closed images. Then $\calC^\vee=(N_i:=(M^i)^\vee,d_i:=^td^{i-1}:N_i\rightarrow N_{i-1})$ is a chain complex
   of abelian groups which are compact and locally separated (with $N_i$ in degree $i$). The $d_i$ are
   strict homomorphisms with closed image, and for all $i$, we have natural isomorphisms
   \begin{equation*}
    \alpha_i:H_i(\calC^\vee)\cong (H^i(\calC))^\vee.
   \end{equation*}
  \end{prop}
  
  In order to be able to apply Proposition~\ref{homcohomduality} to the complex 
  $\calC_{\phi,\gamma_1,\gamma_2}(M)$, we need the following result:-
  
  \begin{lem}
   If $M$ is an \'etale $(\phi,G)$-torsion module over $\AA_K$, then the image of $\phi-1$ contains a
   compact neighbourhood of $0$ on which $\phi-1$ induces a bijection and hence a homeomorphism by
   compactness.
  \end{lem}
  \begin{proof}
   Reduce to the $1$-dimensional case, using Proposition 2.4 in~\cite{herr1}.
  \end{proof} 
  
  \noindent For the rest of this section, let $M=D(\mun)$. Note that as a $\phi$-module, we have $M\cong
  \AA_K\mod p^n$. In particular, this implies that $\psi_M$ is defined.

  \begin{prop}\label{samephi}
   The map $\psi:M^\vee\rightarrow M^\vee: f\rightarrow f\circ\phi_M$ agrees with
   $\psi_{M^\vee}$. 
  \end{prop}
  \begin{proof}
   Imitate the argument in Section 5.5.1 in~\cite{herr2}.
  \end{proof}
  
  Let $0\leq i\leq 3$, and define an isomorphism
  \begin{align*}
   v_i: (M^\vee)^{\oplus{3\choose i}} & \rightarrow (M^{\oplus{3\choose i}})^\vee,\\
        (g_j)_{1\leq j\leq{3\choose i}}& \rightarrow \oplus_{1\leq j\leq{3\choose i}}g_j.
  \end{align*}
  
  \begin{lem}\label{homcohomdual}
   For all $0\leq i\leq 3$, the $v_i$ induce isomorphisms (which we will also denote by $v_i$)
   \begin{equation*}
    H^i(\calC_{\psi_{M^\vee},\gamma_1^{-1},\gamma_2^{-1}}(M^\vee))
    \rightarrow H_{3-i}([\calC_{\phi_M,\gamma_1,\gamma_2}(M)]^\vee).
   \end{equation*}
  \end{lem}
  \begin{proof}
   We only have to show that the actions of $\psi_{M^\vee}$ and $\phi_M$ are compatible with the maps
   $v_i$. But this is shown in Proposition~\ref{samephi}.
  \end{proof}
  
  Combining the isomorphisms of Proposition~\ref{Galcohomology} and Lemmas~\ref{homcohomduality} 
  and~\ref{homcohomdual}, we therefore get for all $0\leq i\leq 3$ an isomorphism
  \begin{align*}
   u_i(M): H^i(\calC_{\phi_{M^\vee},\gamma_1,\gamma_2}(M^\vee)) &\rightarrow 
                                   H^i(\calC_{\psi_{M^\vee},\gamma_1^{-1},\gamma_2^{-1}}(M^\vee))  \\
  			&\rightarrow^{v_i} H_{3-i}([\calC_{\phi_M,\gamma_1,\gamma_2}(M)]^\vee) \\
	 	   &\rightarrow^{\alpha_i} [H^{3-i}(\calC_{\phi_M,\gamma_1,\gamma_2}(M))]^\vee
  \end{align*}
  Using these isomorphisms, we get the following proposition:-
  
  \begin{prop}\label{PontryaginPairing}
   For all $0\leq i\leq 3$, we have a perfect pairing
   \begin{align*}
    H^i(\calC_{\phi_M,\gamma_1,\gamma_2}(M))\times 
    H^{3-i} (\calC_{\phi_{M^\vee},\gamma_1,\gamma_2}(M^\vee))&\rightarrow \QQ_p\slash\ZZ_p,\\
    (x,y)&\rightarrow [u_{3-i}(M)(y)](x).
   \end{align*}
  \end{prop}
  
  Making the above isomorphisms explicit (which is very messy, so we omit the details), one can
  show that
  \begin{align}\label{explicit}
   H^1_{\phi_M,\gamma_1,\gamma_2}(M)\times H^{2}_{\phi_{M^\vee},\gamma_1,\gamma_2}(M^\vee)\rightarrow& 
   \QQ_p\slash\ZZ_p,\\
   (x,y,z)\times(f,g,h)\rightarrow &h(\gamma_2\gamma_1x)-g(\gamma_2\phi_M(y))\\
   &+
                        f(\gamma_1\frac{\gamma_2^{\frac{1}{a}}-1}{\gamma_2-1}\phi_M(z))
		       +g(\tilde{\omega}\phi_M(z)),\\
   H^2_{\phi_M,\gamma_1,\gamma_2}(M)\times H^{1}_{\phi_{M^\vee},\gamma_1,\gamma_2}(M^\vee)\rightarrow &\QQ_p\slash\ZZ_p,\\
   (x,y,z)\times(f,g,h)\rightarrow &-h(\gamma_2x)-
                        g(\gamma_1\frac{\gamma_2^{\frac{1}{a}}-1}{\gamma_2-1}y)-h(\tilde{\omega}y)\\
			&+f(\phi_M(z)),
  \end{align}
  where $\tilde{\omega}\in\Lambda(G)$ is the element satisfying 
  \begin{equation*}
   \tilde{\omega}(\gamma_2^{-1}-1)=n\frac{\gamma_2^{\frac{1}{a}}-1}{\gamma_2-1}-1.
  \end{equation*}


\section{The Tate pairing}\label{thepairing}


 \subsection{Proof of the Tate isomosphism}\label{prooftheorem1}
 
  \noindent {\bf Definition.} Define the map 
  \begin{equation*}
   \TR_K:\Omega^2_{\AA_K}\rightarrow\ZZ_p
  \end{equation*} 
  to be the composition $\Tr_{W(k)\slash\QQ_p}\circ\Res_K$.
  Note that $\TR_K$ is $\ZZ_p$-linear.
  
  \begin{prop}\label{vanishingformulae}
   For all $\omega\in\Omega(K)$, we have 
   \begin{align}
    \TR_K((\gamma_1\frac{\gamma_2^{\frac{1}{n}}-1}{\gamma_2-1}-1)\omega)&=0, \label{actiongamma1}\\
    \TR_K((\gamma_2-1)\omega)&=0, \label{actiongamma2}\\
    \TR_K((\phi_{\Omega(K)}-1)\omega)&=0.  \label{actionphi}
   \end{align}
  \end{prop}
  \begin{proof}
   Since $\gamma_2(T)=(1+\pi)^NT$ from some $N\in\ZZ_p^\times$, it is clear that~\eqref{actiongamma2}
   holds. Write $\omega=\sum_{i\in\ZZ}a_iT^idT\wedge d\pi_F$, where $a_i\in\AA_F$ for all $i\in\ZZ$. 
   As shown in~\cite{herr2}, we have $\TR_K((\phi-1) a_0 d\pi_F)=0$, which (by the compatibilities of the
   actions of $\phi$) implies~\eqref{actionphi}. It remains to show~\eqref{actiongamma1}. Let
   $x=\gamma_2-1$. Expanding $\frac{\gamma_2^{\frac{1}{a}}-1}{\gamma_2-1}$ in terms of $x$ gives
   \begin{equation*}
    \frac{\gamma_2^{\frac{1}{a}}-1}{\gamma_2-1}=\frac{1}{a}+ 
    \text{higher order terms}.
   \end{equation*}
   Since $\gamma_2$ is trivial on $\AA_F$, the operator 
   $\gamma_1\frac{\gamma_2^{\frac{1}{a}}-1}{\gamma_2-1}-1$ acts as $\gamma_1-1$ on $a_0$.
   Equation~\eqref{actiongamma1} therefore follows from the $1$-dimensional case as treated
   in~\cite{herr2}.
  \end{proof}
  
  For $j\geq 1$, let $\Omega_j(K)=\Omega(K)\mod p^j$, which is an \'etale $(\phi,G)$-torsion module over
  $\AA_K$ isomorphic to $D_K(\mu_{p^j})$. 
  By reduction $\mod p^j$, $\Res$ induces a canonical $\ZZ\slash p^j\ZZ$-linear map
  \begin{equation*}
   \TR_{K,j}: \Omega_j(K)\rightarrow \ZZ\slash p^j\ZZ.
  \end{equation*}
  By Proposition~\ref{vanishingformulae}, $\TR_{K,j}$ factorizes through the quotient of $\Omega_j(K)$ by
  $(\tau_1(\Omega_j(K)+\tau_2(\Omega_j(K)+(\phi-1)(\Omega_i(K))$, where $\tau_i=\gamma_i-1$. 
  We therefore get a homomorphism
  \begin{equation*}
   \TR_{K,j}:H^3_{\phi,\gamma_1,\gamma_2}(\Omega_i(K))\rightarrow \ZZ\slash p^j\ZZ.
  \end{equation*}
  Passing to the direct limit gives a map 
  \begin{equation*}
   H^3_{\phi,\gamma_1,\gamma_2}(\BB_K\slash\AA_K\otimes_{\AA_K}\Omega(K))\rightarrow
   \QQ_p\slash\ZZ_p.
  \end{equation*}
  Let $v_p$ be the $p$-adic valuation of $\ZZ_p$ normalized by $v_p(p)=1$. Let 
  $n_p(\gammaone)=v_p(\log\chi(\gamma_1))$ and
  $n_p(\gammatwo)=v_p(\eta(\gammatwo))$. 
  Let $c=\frac{p^{n_p(\gammaone)}}{\log \chi(\gammaone)}\frac{p^{n_p(\gammatwo)}}{\eta(\gammatwo)}$.
  
  \begin{prop}
   The map $-c\TR$ gives a canonical isomorphism between the groups
   $H^3_{\phi,\gamma_1,\gamma_2}(\BB_K\slash\AA_K\otimes_{\AA_K}\Omega(K))$ and $\QQ_p\slash\ZZ_p$.
  \end{prop}
  \noindent {\bf Remark.} The factor $c$ may seem bizarre, but we will see its use in
  Section~\ref{higherhilbert}.
  \begin{proof}
   It is sufficient to show that $\TR_{K,j}$ gives an isomorphism for all $j$. Since we can expand any
   element in $\Omega_j(K)$ as a power series in $T$ with coefficients in $\AA_F\mod p^j$, it is
   sufficient to show that for all $k\neq 0$, $\gamma_2-1$ is surjective on $T^k\AA_F$. The proposition
   follows from Herr's proof of the Tate isomorphism in the $1$-dimensional case (c.f. Theorem 5.2
   in~\cite{herr2}). Expanding the power series shows that 
   $(\gamma_2-1)(T^k)=T^kf_k(\pi)$, where $f_k(\pi)\in\AA_F^\times$, which finishes the proof.
  \end{proof}
  
  Combining this result with Proposition~\ref{isomtotwist} and the main result of Section~\ref{galcohom}
  proves Theorem~\ref{Tateisom}.


 \subsection{Relation to Pontryagin duality}\label{pontryagin}
 
  Let $V$ be a torsion $\ZZ_p$-representation of $\calG_K$, and let $M=D(V)$. Recall that
  \begin{equation*}
   \tilde{V}=\Hom_{\Zp}(V,\mup).
  \end{equation*}
  
  \begin{lem}
   The $(\phi,G)$-module $D_K(\tilde{V})$ is isomorphic to 
   $\Hom_{\AA_K}(M,\BB_K\slash\AA_K\otimes_{\AA_K}\Omega(K))$.
  \end{lem}
  \begin{proof}
   It follows from the equivalence of categories (Theorem~\ref{eqofcat}) that $D_K(\tilde{V})$ is isomorphic
   to $\Hom_{\AA_K}(M,D(\mup^{\otimes n}))$. It is now sufficient to observe that
   \begin{equation*}
    D_K(\mu_{p^\infty}^{\otimes n})=\varinjlim D_K(\mu_{p^j}^{\otimes n})\cong\varinjlim
    \frac{\Omega(K)}{p^j\Omega(K)}
    \cong \BB_K\slash\AA_K\otimes_{\AA_K}\Omega(K).
   \end{equation*}
  \end{proof}
  
  Theorem~\ref{Duality} is therefore a consequence of Lemma~\ref{residueisom} and
  Theorem~\ref{PontryaginPairing}.

 
\section{The higher Hilbert pairing}\label{higherhilbert}

 \subsection{Construction of the pairing}
 
  Let $F$ be the maximal algebraic extension of $\QQ_p$ contained in $K$. 
  Throughout this section, we assume that the extension of $F$ over $\QQ_p$ is unramified. 
  As in the previous sections, let $\epsilon=(1,\xi_p,\xi_{p^2},\dots)\in\EE_K$ and
  $\pi=[\epsilon]-1$, where $[\epsilon]$ is the Teichm\"uller representative of $\epsilon$. Also, let
  $\fT=(X,X^{\frac{1}{p}},X^{\frac{1}{p^2}},\dots)\in\EE_K$, and let $T=[\fT]\in\AA_K$ be its
  Teichm\"uller representative.
  \vs
  
  Fix $n\geq 1$. Let $\pi_n=\phi^{-n}(\pi)=[(\xin,\xi_{p^{n+1}}\dots)]-1$ and 
  $T_n=\phi^{-n}(T)=[(X^{\frac{1}{p^n}},X^{\frac{1}{p^{n+1}}},\dots)]$.
 
  \begin{prop}\label{cupisom}
   Let $n\geq 1$. Then for all $i\geq 0$, taking cup product with $\xin$ gives an isomorphism of
   $\Gal(K_n\slash K)$-modules
   \begin{equation*}
    \cup\xin: H^i(G_{K_n},\ZZ\slash p^n\ZZ)\rightarrow H^i(G_{K_n},\mun).
   \end{equation*}
  \end{prop}
  \begin{proof}
   Let $D=D(\ZZ\slash p^n\ZZ)\cong\AA_{K_n}\mod p^n$.
   Let $\Gamma_1^{(n)}=\Gal(K(\mup,X^{\frac{1}{p^n}})\slash K_n)$ and
   $\Gamma_2^{(n)}=\Gal(K_\infty\slash K(\mup,X^{\frac{1}{p^n}}))$. 
   Let $\gammaone$ and $\gammatwo$ be topological generators of $\Gamma_1^{(n)}$ and
   $\Gamma_2^{(n)}$, respectively.
   Recall that the $G_{K_n}$-cohomology of $\ZZ\slash p^n\ZZ$ is given by the complex
   \begin{equation*}
    \calC_{\phi,\gamma_{1,n},\gamma_{2,n}}(D): 0  \rTo D \rTo^{f_1} D^{\oplus 3} \rTo^{f_2}   
    D^{\oplus 3}\rTo^{f_3} D \rTo 0, 
   \end{equation*}
   where 
   \begin{align*}
    f_1: x      \rightarrow & [ (\phi-1)x,(\gammaone-1)x,(\gammatwo-1)x],\\
    f_2: (x,y,z)\rightarrow & [ (\phi-1)y-(\gammaone-1)x,\\
                            & (\phi-1)z-(\gammatwo-1)x,\\
	 		   & (\gammatwo-1)y-(\gammaone\frac{\gammatwo^{\frac{1}{a}}-1}
			   {\gammatwo-1}-1)z ],\\
    f_3: (x,y,z)\rightarrow & (\gammatwo-1)x-(\gammaone\frac{\gammatwo^{\frac{1}{a}}-1}
    {\gammatwo-1}-1)y-(\phi-1)z.
   \end{align*}
   Since $\mun=<\xin>$, it is easy to see from this description of the cohomology groups that cup product
   (which is the same as multiplication) with $\xin$ gives a $\Gal(K_n\slash K)$-equivariant isomorphims
   $H^i(G_{K_n},\ZZ\slash p^n\ZZ)\rightarrow H^i(G_{K_n},\xin)$ for all $0\leq i\leq 3$.
  \end{proof}
  
  By Theorem~\ref{Duality}, we have a perfect pairing
  \begin{equation*}
   H^2(G_{K_n},\mun^{\otimes 2})\times H^1(G_{K_n},\ZZ\slash p^n\ZZ)\rightarrow \QQ_p\slash\ZZ_p.
  \end{equation*}
  Taking cup product with $\xin$ and using Proposition~\ref{cupisom} gives a perfect pairing
  \begin{equation}\label{pairingcohom}
   H^2(G_{K_n},\mun^{\otimes 2})\times H^1(G_{K_n},\mun)\rightarrow\mun.
  \end{equation}
  
  \noindent {\bf Definition.} The Hilbert pairing 
  \begin{equation*}
   K_2(K_n)\times K_1(K_n)\rightarrow \mun.
  \end{equation*}
  is the composition of~\eqref{pairingcohom} with the Galois symbol map
  \begin{equation*}
   \delta^2\times\delta: K_2(K_n)\times K_1(K_n)\rightarrow 
   H^2(G_{K_n},\mun^{\otimes 2})\times H^1(G_{K_n},\mun)
  \end{equation*}
  
  Note that we have a natural (surjective) multiplication map 
  $K_1(K_n)\times K_1(K_n)\rightarrow K_2(K_n)$. We therefore interprete the Hilbert pairing as a pairing
  \begin{equation}\label{threetimesK1}
   V_n:K_1(K_n)\times K_1(K_n)\times K_1(K_n)\rightarrow \mun.
  \end{equation}
  From the definition of the Galois symbol it is clear that $V_n$ factors through
  \begin{equation}\label{threetimesK1*}
   \frakV_n:\big(K_1(K_n)\slash p^n\big)^{\times 3}\rightarrow \mun.
  \end{equation}
  Since the pairing~\eqref{pairingcohom} is perfect, the pairing $\frakV_n$ in~\eqref{threetimesK1*} is
  non-degenerate.
  \vs

  \begin{lem}\label{galoissymbols}
   We have a commutative diagram
   \begin{diagram}
    K_1(K_n)\times K_1(K_n) & \rTo^{\delta\times\delta} & H^1(K_n,\mun)\times H^1(K_n,\mun)\\
    \dTo                    &                           &                \dTo^{\cup} \\
    K_2(K_n)                & \rTo^{\delta^2}           & H^2(K_n,\mun)\\
   \end{diagram}
  \end{lem}



  \subsection{The Kummer map}\label{Kummer}
   
   Let $F$ by the maximal extension of $\QQ_p$ contained in $K$, and let $R=O_F[[T_n]][T_n^{-1}]$.
   Note that we can identifiy $\AA_{K_n}^+$ with the abstract power series ring $R[[Y]]$ (where
   $Y=\pi_n$, but we forget this for the time being). 
   Let $\fm=(p,Y)$ be the maximal ideal of $\AA_{K_n}^+$, and let $\calA=1+\fm$. 
   For $F(Y)\in\calA$, define
   \begin{equation*}
    l(F(Y))=(1-\frac{\phi}{p})\log F(Y).
   \end{equation*}
   To shorten notation, let $f(Y)=l(F(Y))$. Define the differential operators 
   $D_1=(Y+1)\frac{d}{dY}$ and $D_2=T_n\frac{d}{dT_n}$. Fix $n\geq 1$, and let $S_n=R[[\pi_n]]$.
   \vs
      
   \noindent {\bf Note.} (1) The action of $\gamma_2$ on $\AA_{K_n}^+$ is trivial $\mod\pi$.
   
   \noindent (2) The element $\epsilon$ is a generator of $\ZZ_p(1)$, so $\epsilon\mod p^n$ is a 
   generator of $\mun$ and can be identified with $\xin$. 
   \vs
   
   Let $\epsilon^{(n)}=\epsilon\mod p^n$ and $\tau=\frac{1}{\pi}-\frac{1}{2}$.
   
   \begin{prop}\label{congruences}
    Let $F(Y)\in\calA$. Then there exist unique $\agamma(\pi_n),\bgamma(\pi_n)\in S_n$ such that 
    \begin{align*}
     (\phi-1)(\agamma(\pi_n)\otimes\epsilon^{(n)})&=(\gammaone-1)(f(\pi_n)\tau\otimes\epsilon^{(n)}),\\
     (\phi-1)(\bgamma(\pi_n)\otimes\epsilon^{(n)})&=(\gammatwo-1)(f(\pi_n)\tau\otimes\epsilon^{(n)}).
    \end{align*}
    Moreover, we have
    \begin{align*}
     \agamma(\pi_n)&=\frac{1-\chi(\gammaone)}{p^n}D_1\log F(\pi_n)\mod\pi,\\   
     \bgamma(\pi_n)&=\eta_n(\gammatwo)D_2\log F(\pi_n)\mod\pi.
    \end{align*}
   \end{prop} 
   \begin{proof}   
    Arguing as in the proof of Lemma 2.1.3 in~\cite{benois1} we have
    \begin{equation*}
     (\gammaone-1)(f(\pi_n)\tau\otimes\epsilon^{(n)})=-\frac{1-\chi(\gammaone)}{p^n}D_1\log 
     f(\pi_n)\otimes\epsilon^{(n)}\mod\pi.
    \end{equation*}
    Using the identity $D_1\phi=p\phi D_1$, we can write
    \begin{equation*}
     (\gammaone-1)(f(\pi_n)\tau\otimes\epsilon^{(n)})=(\phi-1)(\frac{1-\chi(\gammaone)}{p^n}D_1\log 
     F(\pi_n)\otimes\epsilon^{(n)})\mod\pi.
    \end{equation*}
    Let $\tagamma(\pi_n)=\frac{1-\chi(\gammaone)}{p^n}D_1\log 
    F(\pi_n)$, so 
    \begin{equation*}
     (\phi-1)(\tagamma(\pi_n)\otimes\epsilon^{(n)})=(\gammaone-1)(f(\pi_n)\tau\otimes\epsilon^{(n)}).
    \end{equation*}
    Since $\phi-1$ is invertible on $\pi S_n$, we deduce that there exists a unique $\agamma(\pi_n)\in S_n$ such
    that $\agamma(\pi_n)=\tagamma(\pi_n)\mod\pi$ and 
    $(\phi-1)(\agamma(\pi_n)\otimes\epsilon^{(n)})=(\gammaone-1)(f(\pi_n)\tau\otimes\epsilon^{(n)}).$
    \vs
    
    \noindent The existence of $\bgamma(\pi_n)$ follows from similar arguments:- Let
    $\eta_n(\gammatwo)=\frac{\eta(\gammatwo)}{p^n}$.
    Note that 
    \begin{equation}\label{actiononT}
     \gammatwo(T_n)=(1+\pi)^{\eta_n(\gammatwo)}T_n,
    \end{equation}
    so
    \begin{equation*}
     \gammatwo(T_n)=T_n+\eta_n(\gammatwo)T_n\pi\mod \pi^2,
    \end{equation*}
    and hence
    \begin{align*}
     \gammatwo f(\pi_n)&=f(\pi_n)+\eta_n(\gammatwo)D_2f(\pi_n)\pi\mod\pi^2,\\
     (\gammatwo-1) (f(\pi_n)\tau\otimes\epsilon^{(n)})&=\eta_n(\gammatwo)D_2f(\pi_n)\otimes
     \epsilon^{(n)}\mod\pi.
    \end{align*}
    By assumption we have
    \begin{equation*}
     f(\pi_n)=(1-\frac{\phi}{p})\log F(\pi_n).
    \end{equation*}
    Since $D_2\phi=p\phi D_2$, it follows that
    \begin{equation*}
     (\gammatwo-1) f(\pi_n)\tau\otimes\epsilon^{(n)}=(1-\phi)\eta_n(\gammatwo)D_2\log F(\pi_n)
     \otimes\epsilon^{(n)}\mod\pi.
    \end{equation*}
    Let $\tbgamma(\pi_n)=\eta_n(\gammatwo)D_2\log F(\pi_n)\otimes\epsilon^{(n)}$.
    It follows that $(1-\phi)\tbgamma(\pi_n)=(\gammatwo-1)f(\pi_n)\tau\otimes\epsilon^{(n)}$. Since
    $\phi-1$ is invertible on $\pi S_n$, there exists a unique $\bgamma(\pi_n)\in S_n$ such that
    $\bgamma=\tbgamma\mod\pi$ and
    \begin{equation*}
     (\phi-1)\bgamma(\pi_n)\otimes\epsilon^{(n)}=(\gammatwo-1)(f(\pi_n)\tau\otimes\epsilon^{(n)}).
    \end{equation*}
   \end{proof} 
       
   \noindent {\bf Definition.} Let $\iota_n:\calA\rightarrow H^1_{\phi,\gammaone,\gammatwo}(C_n)$ be the
   homomorphism
   \begin{equation}
    F(X)\rightarrow [f(\pi_n)\tau\otimes \epsilon^{(n)},\agamma(\pi_n)\otimes\epsilon^{(n)},
    \bgamma\otimes\epsilon^{(n)}].
   \end{equation}
   
   \begin{lem}
    The map $\iota_n$ is well-defined.
   \end{lem}
   \begin{proof}
    Explicit calculation shows that
    \begin{equation}
     (\gammatwo-1)(\agamma(\pi_n)\otimes\epsilon^{(n)})=
     (\gammaone\frac{\gammatwoN-1}{\gammatwo-1}-1)(\bgamma(\pi_n)\otimes\epsilon^{(n)}).
    \end{equation}
    It follows that $\iota_n$ is really a map into $H^1_{\gammaone,\gammatwo,\phi}$.
   \end{proof}
   
   \begin{prop}\label{descriptionkummer}
    Let $\delta_n:K_n^\times\rightarrow H^1(G_{K_n},\ZZ_p(1))$ be the Kummer map.
    We have a commutative diagram
    \begin{diagram}
     \calA & \rTo^{\iota_n}& H^1_{\phi,\gammaone, \gammatwo}(C_n)\\
     \dTo^{h_n} &       &\dTo^{\cong}\\
     K_n^\times & \rTo^{\delta_n} & H^1(G_{K_n},\ZZ_p(1))
    \end{diagram}
   \end{prop}
   
   To prove the proposition, we follow the strategy of Benois in the proof of Proposition 2.1.5 
   in~\cite{benois1}.
   We split the proof of the proposition into a sequence of lemmas. 
   \vs
   
   Note that the action of $\calG_K$ on $\AA_K$ factors through $G_K=\Gal(K_\infty\slash K)$. 
   Recall that $G_K\cong \Gamma_1\rtimes\Gamma_2$, where $\Gamma_1$ is congruent (via the cyclotomic
   character $\chi_K$) to an open subgroup of $\ZZ_p^\times$ and $\Gamma_2$ is congruent (via a character
   $\eta_K$) to $\ZZ_p$. 
   
   \begin{lem}
    Let $[x,y,z]\in H^1_{\gammaone,\gammatwo,\phi}(C_n)$, and let $u\in\AA$ be a solution of
    $(\phi-1)u=x$. Then $h^1([x,y,z])$ is given by the cocycle $\sigma\rightarrow c(\sigma)$ which is
    defined as follows:- 
    Let $\tsigma$ be the image of $\sigma$ in $G_K$ under the projection map. Let $k=\chi(\tsigma)$ and
    $l=\eta(\tsigma)$, so $\tsigma=\gammaone^k\gammatwo^l$. Then
    \begin{equation*}
     c(\sigma)=(\sigma-1)u-\frac{\gammaone^k-1}{\gammaone-1}y-\gammaone^k
     \frac{\gammatwo^l-1}{\gammatwo-1}z.
    \end{equation*}
   \end{lem}
   \begin{proof}
    Let $N_{x,y,z}=D(\mun)\oplus \AA_{K_n}e$, where the action of $\phi,\gammaone$ and $\gammatwo$ on $e$
    is given by $\phi(e)=e+x$, $\gammaone(e)=e+y$, $\gammatwo(e)=e+z$. Then the long exact sequence
    associated to the short exact sequence of $G_{K_n}$-modules
    \begin{equation}\label{sesauxilary}
     0\rightarrow D(\mun)\rightarrow N_{x,y,z}\rightarrow \AA_{K_n}\rightarrow 0
    \end{equation}
    gives the connecting homomorhism $\delta: H^0(\AA_{K_n})\rightarrow 
    H^1_{\gammaone,\gammatwo,\phi}(C_n)$, and an easy diagram search shows that $\delta(1)=[x,y,z]$.
    Applying $(\phi-1)$ to~\eqref{sesauxilary} gives a short exact sequence
    \begin{equation*}
     0\rightarrow \mun\rightarrow T_{x,y,z}\rightarrow \ZZ_p\rightarrow 0
    \end{equation*}
    and a connecting homomorphism $\delta_{\Gal}:\ZZ_p\rightarrow H^1(G_K,\mun)$. We have $u+e\in
    N_{x,y,z}\otimes_{\AA_{K_n}}\AA$ and $(\phi-1)(u+e)=0$, so $u+e\in T_{x,y,z}$. So $\delta_{\Gal}(1)$
    can be represented by the cocycle 
    \begin{align*}
     \sigma\rightarrow& \sigma (u+e) -(u-e)\\
                     =& (\sigma-1)u +(\sigma-1)e.
    \end{align*}
    Now $\calG_K$ acts on $\AA_{K_n}$  via the quotient $G_{K_n}$. Since $e\in\AA_{K_n}$, we have
    \begin{align*}
     (\tsigma-1)e &= (\gammaone^k\gammatwo^l-1)e\\
                  &= \gammaone^k(\gammatwo^l-1)e +(\gammaone^k-1)e\\
		  &= \gammaone^k\frac{\gammatwo^l-1}{\gammatwo-1}z+\frac{\gammaone^k-1}{\gammaone-1}y
    \end{align*}
    The lemma now follows from the commutativity of the diagram
    \begin{diagram}
     H^0(\AA_{K_n}) & \rTo^{\delta} & H^1_{\gammaone,\gammatwo,\phi}(C_n) \\
     \dTo^{h^0}     &               & \dTo^{h^1} \\
     \ZZ_p          & \rTo^{\delta_{\Gal}} & H^1(\calG_{K_n},\mun)
    \end{diagram}
   \end{proof}    
   
   In particular, $h^1(\iota_n(F))$ is given by
   \begin{align*}
    \sigma\rightarrow &(\sigma-1)u-\frac{\gammaone^{\chi(\sigma)}-1}{\gammaone-1}\agamma(\pi_n)\otimes\epsilon^{(n)}\\
                      &-\gammaone^{\chi(\sigma)}\frac{\gammatwo^{\eta(\sigma)}-1}
		      {\gammatwo-1}\bgamma(\pi_n)\otimes\epsilon^{(n)},
   \end{align*}
   where $(1-\phi)u=f(\pi_n)\tau$. Since $\gammaone(\pi_n)=\pi_n\mod\pi$ and $\gammatwo(T_n)=T_n\mod\pi$, we
   have
   \begin{align*}
    \frac{\gammaone^{\chi(\sigma)}-1}{\gammaone-1}\agamma(\pi_n)\otimes\epsilon^{(n)}&\cong
    \chi(\sigma)\frac{1-\chi(\sigma)}{p^n}D_1\log F(\pi_n)\otimes\epsilon^{(n)}\mod\pi,\\
    \gammaone^{\chi(\sigma)}\frac{\gammatwo^{\eta(\sigma)}-1}{\gammatwo-1}
    \bgamma(\pi_n)\otimes\epsilon^{(n)} &\cong \eta(\sigma)\eta_n(\gammatwo)D_2\log F(\pi_n)\mod\pi.
   \end{align*}
   These congruences imply that
   \begin{equation*}
    c(\sigma)\cong (\chi(\sigma)\sigma-1)u+\frac{1-\chi(\sigma)}{p^n}D_1\log F(\pi_n)
    +\frac{\eta(\sigma)}{\chi(\sigma)}\eta_n(\gammatwo)D_2\log F(\pi_n)\mod\pi.
   \end{equation*}
   We now interprete $c(\sigma)$ in terms of $\AA_{\cris}$. Denote by $I$ the ideal of $\AA_{\cris}$
   generated by $\pi^2$ and $\frac{\pi^{p-1}}{p}$.
   
   \begin{lem}
    There exists a unique $x\in\Fil^1\AA_{\cris}$ such that $x=u(\pi-\frac{\pi^2}{2})\mod I$ and
    \begin{equation*}
     (1-\frac{\phi}{p})x=f(\pi_n).
    \end{equation*}
   \end{lem}
   \begin{proof}
    Imitate the proof of Lemma 2.1.6.2 in~\cite{benois1}.
   \end{proof}
    
   Define the element 
   \begin{equation*}
    \mu(\sigma)=(\sigma-1)x-\log (\frac{\sigma(F(\pi_n))}{F(\pi_n)}),
   \end{equation*}
   which belongs to $1+\Fil^1\AA_{\cris}$ for all $\sigma\in\calG_{K_n}$, and it is easy to check that the
   map $\mu:\calG_{K_n}\rightarrow \Fil^1\AA_{\cris}$ is a cocycle. 
   
   \begin{lem}
    We have $\mu(\sigma)=c(\sigma)$.
   \end{lem}
   \begin{proof}
    We have $(1-\frac{\phi}{p})\mu(\sigma)=0$, so $\mu(\sigma)$ has the form $a(\sigma)t$ for some
    $a(\sigma)\in\QQ_p$. On the other hand, from the congruences
    \begin{align*}
     \tsigma F(\pi_n)=&F(\pi_n)+\chi(\sigma)\frac{1-\chi(\sigma)}{p^n}D_1\log F(\pi_n)\pi \\
     &+\eta(\sigma)\eta_n(\gammatwo)D_2\log F(\pi_n)\pi\mod\pi^2
    \end{align*}
    and $(\sigma-1)x=(\chi(\sigma)\sigma-1)u\pi\mod I$ it follows that
    \begin{align*}
     \mu(\sigma)=&
    (\chi(\sigma)\sigma-1)u\pi+\frac{1-\chi(\sigma)}{p^n}D_1\log F(\pi_n)\pi \\
    &+\frac{\eta(\sigma)}{\chi(\sigma)}\eta_n(\gammatwo)D_2\log F(\pi_n)\pi\mod I\\
    =& c(\sigma)T\mod I,
   \end{align*}
   which implies that $\mu(\sigma)=tc(\sigma)$.
  \end{proof}
  
  \begin{cor}
   One has
   \begin{equation*}
    [\epsilon]^{c(\sigma)}=\exp(\mu(\sigma))=\frac{\sigma\exp(x)}{\exp(x)}
    \frac{F(\pi_n)}{\sigma F(\pi_n)}.
   \end{equation*}
  \end{cor}
  
  \noindent {\it Proof of Proposition~\ref{descriptionkummer}.} 
  Let $y=\exp(x)$. Then the equation $(1-\frac{\phi}{p})x=f(\pi_n)$ can be written of the form
  \begin{equation*}
   \frac{y^p}{\phi(y)}=\exp(pf(\pi_n)).
  \end{equation*}
  Consider the short exact sequence
  \begin{equation*}
   1\rightarrow [\epsilon]^{\ZZ_p}\rightarrow 1+W^1(R)\rightarrow^{\nu}1+pW(R)\rightarrow 1,
  \end{equation*}
  where $\nu(a)=\frac{a^p}{\phi(a)}$. It shows that the inclusion $W(R)\subset\AA_{\cris}$ gives a $1-1$
  correspondence between solutions $Y$ of $\frac{Y^p}{\phi(Y)}=\exp(pf(\pi_n))$ and solutions $X=\log Y$
  of $(1-\frac{\phi}{p})X=f(\pi_n)$. Hence $y\in 1+W^1(R)$, and it is easy to see by induction that
  \begin{equation*}
   \frac{y^{p^n}}{\phi^n(y)}=\frac{F(\pi_n)^{p^n}}{\phi^n(F(\pi_n))}.
  \end{equation*}
  Let $z=\phi^{-n}(yF(\pi_n)^{-1})$. Applying the map $\theta:W(R)\rightarrow O_{\CC_p}$ to both sides of
  this equation, we obtain that 
  \begin{equation*}
   \theta(z)^{p^n}=h_n(F)^{-1}.
  \end{equation*}
  Hence the connecting map $\delta_n$ sends $h_n(F)$ to the class of the cocycle $\sigma\rightarrow \theta
  (z/\sigma(z))$. On the other hand, one has
  \begin{equation*}
   \theta(\frac{z}{\sigma(z)})=\theta\phi^{-n}(\frac{y\sigma F(\pi_n)}{\sigma(y)F(\pi_n)})
   =\theta\phi^{-n}([\epsilon]^{-c(\sigma)})=\xin^{-c(\sigma)},
  \end{equation*}
   which finishes the proof.
   
    
 \subsection{Vostokov's formulae}\label{vostokovsformulae}

 
  In this section we reprove Vostokov's formulae for~\eqref{threetimesK1}. More precisely, we prove the
  following result:- For $1\leq i\leq 3$, let $\alpha_i\in O_K^\times$ such that $\alpha_i\cong 1\mod
  \bpi_n$, and let $F_i(X)\in \AA_K^+$ such
  that $h_n(F_i)=\alpha_i$. Let $f_i(X)=(1-\frac{\phi}{p})\log F(X)$.
  
  \begin{thm}\label{formulaehilbertspecial}
   We have 
   \begin{equation}
    V_n(\alpha_1,\alpha_2,\alpha_3)=\mun^{\Tr \Res_{\pi_n,T}(\Phi)},
   \end{equation}
   where $\Phi$ is given by the formula
   \begin{align*}
    \Phi=&-\frac{1}{\pi} (\frac{1}{p^2}f_1(\pi_n)d\log F_2^{\phi}(\pi_n) \wedge d\log F_3^\phi(\pi_n)\\
         & -\frac{1}{p}f_2(\pi_n)d\log F_1(\pi_n) \wedge d\log F_3^\phi(\pi_n)\\
         & +f_3(\pi_n)d\log F_1(\pi_n)\wedge d\log F_2(\pi_n)).
   \end{align*}
  \end{thm}
  
  We will prove the theorem in the rest of this section.
  
  \begin{lem}
   Let $[x,y,z],[x',y',z']\in H^1_{\phi,\gammaone, \gammatwo}(C_n)$. If $[\calx, \caly,\calz]$ represents
   the cohomology class of $[x,y,z]\cup [x',y',z']$, then
   \begin{align*}
    \calx&=y\otimes \gammaone x'-x\otimes\phi y',\\
    \caly&=z\otimes\gammatwo x'-x\otimes\phi z'.
   \end{align*}
   Moreover, if $z,z'\in S_n$, then 
   \begin{equation*}
    \calz= x\otimes \gammatwo y'-y\otimes\gammaone x'\mod\pi.
   \end{equation*}
  \end{lem}
  \begin{proof}
   Since $\Gamma_1^{(n)}$ (resp. $\Gamma_2^{(n)}$) is isomorphic to an open subgroup of $\ZZ_p^\times$
   (resp. $\ZZ_p$), the formulae for $\calx$
   and $\caly$ follow from~\cite{herr2}, using that the cup product is compatible with restriction. The
   formula for $\calz$ follows from the observation that $\gamma_1$ and $\gamma_2$ commute on
   $S_n\mod\pi$.
  \end{proof}
  
  For $1\leq i\leq 3$, let $\iota_n(F_i(X))= 
  [f_i(\pi_n)\tau\otimes \epsilon,\agamma^{(i)}(\pi_n)\otimes\epsilon,\bgamma^{(i)}\otimes\epsilon]$.

  \begin{cor}\label{cupformulae}
   If $[\calx,\caly,\calz]=\iota_n(F_1(X))\cup\iota_n(F_2(X))$, then 
   \begin{align*}
    \calx=& \frac{1-\chi(\gammaone)}{p^n} (D_1\log F_1(\pi_n)\otimes 
    f_2(\pi_n)\\
          &-p^{-1}f_1(\pi_n)\otimes D_1\log F_2^\phi(\pi_n))\otimes\tau\otimes\epsilon^2\mod S_n,\\
    \caly=&\eta_n(\gammatwo)(D_2\log F_1(\pi_n)\otimes f_2(\pi_n)\\
          &-p^{-1}f_1(\pi_n)\otimes D_2\log F_2^\phi(\pi_n))\otimes\tau\otimes\epsilon^2\mod S_n,\\
    \calz=& \frac{1-\chi(\gammaone)}{p^n}\eta_n(\gammatwo)
           (D_1\log F_2(\pi_n)D_2\log F_1(\pi_n)\\
          &-D_2\log F_1(\pi_n)D_1\log F_2(\pi_n))\otimes\epsilon^2\mod\pi
   \end{align*}    
  \end{cor}
  \begin{proof}
   Observe that
   \begin{align*}
    \gammaone(f_i(\pi_n)\tau)&=\chi^{-1}(\gammaone)f_i(\pi_n)\tau\mod S_n, \\
    \gammatwo(f_i(\pi_n)\tau)&= f_i(\pi_n)\tau\mod S_n.
   \end{align*}
   The lemma now follows from the previous lemma and Proposition~\ref{congruences}.
  \end{proof}
  
  \noindent {\it Proof of Theorem~\ref{formulaehilbert}.} We prove the Theorem using the 
  formulae~\eqref{explicit}. Let 
  \begin{equation*}
   H_{\alpha_1,\alpha_2,\alpha_3}=\iota_n(F_3(X))\cup [\calx,\caly,\calz],
  \end{equation*}
  where (to simplify the notation) we write
  \begin{equation*}
   [x,y,z]=[f_3(\pi_n)\tau,\agamma^{(3)}(\pi_n),\bgamma^{(3)}].
  \end{equation*}
  Recall that $\gammaone$ and $\gammatwo$ commute on
  $S_n\mod\pi$. It follows that the formulae~\eqref{explicit} simplify to 
  \begin{equation*}
   \iota_n(F_3(X))\cup [\calx,\caly,\calz]= \calz\otimes\gamma_2\gamma_1(x)-
   \caly\otimes\gamma_2\phi_M(y)+\calx\otimes\gamma_1\phi_M(z) \mod S_n.
  \end{equation*}
  Using the formulae in Corollary~\ref{cupformulae} it follows that
  \begin{align*}
   H_{\alpha_1,\alpha_2,\alpha_3}=\gamma_1\phi_M(\bgamma^{(3)}(\pi_n))&
   (\frac{1-\chi(\gammaone)}{p^n} (D_1\log F_1(\pi_n)\otimes f_2(\pi_n)\\
         &-p^{-1}f_1(\pi_n)\otimes D_1\log F_2^\phi(\pi_n))\otimes\tau\otimes(\epsilon^{(n)})^2)\\
   -\gamma_2\phi_M(\agamma^{(3)}(\pi_n))&(\eta_n(\gammatwo)(D_2\log F_1(\pi_n)\otimes f_2(\pi_n)\\
         &-p^{-1}f_1(\pi_n)\otimes D_2\log F_2^\phi(\pi_n))\otimes\tau\otimes(\epsilon^{(n)})^2)\\
  +\gamma_2\gamma_1 (f_3(\pi_n)\tau)&(\frac{1-\chi(\gammaone)}{p^n}\eta_n(\gammatwo)
          (D_1\log F_2(\pi_n)D_2\log F_1(\pi_n)\\
         &-D_2\log F_1(\pi_n)D_1\log F_2(\pi_n))\otimes(\epsilon^{(n)})^2)\mod S_n
  \end{align*}    
  As before, we have
  \begin{align*}
   \gammaone(f_i(\pi_n)\tau)&=\chi^{-1}(\gammaone)f_i(\pi_n)\tau\mod S_n, \\
   \gammatwo(f_i(\pi_n)\tau)&= f_i(\pi_n)\tau\mod S_n,
  \end{align*}
  so (rearranging the terms) the above formula simplifies to 
  \begin{align*}
   H_{\alpha_1,\alpha_2,\alpha_3}=&\frac{1-\chi(\gammaone)}{p^n}\eta_n(\gammatwo)\pi^{-1} \times  \\
   -p^{-2} f_1(\pi_n)&(D_1\log F_2^\phi(\pi_n)D_2\log F_3^\phi(\pi_n)\\
   &+D_2\log F_2^\phi(\pi_n)D_1\log F_3^\phi(\pi_n)))\\
   +p^{-1} f_2(\pi_n)&(D_1\log F_1(\pi_n)D_2\log F_3^\phi(\pi_n)\\
   &+D_2\log F_1(\pi_n) D_1\log F_3^\phi(\pi_n))\\
   -f_3(\pi_n)&(D_1\log F_2(\pi_n)D_2\log F_1(\pi_n)\\
   &+D_2\log F_1(\pi_n)D_1\log F_2(\pi_n))\otimes(\epsilon^{(n)})^2\mod S_n.
  \end{align*}
  Recall that $D_1=(\pi_n+1)\frac{d}{d\pi_n}$ and $D_2=T_n\frac{d}{dT_n}$. It follows that the image in
  $\Omega(K)$ of the above expression (which we also denote by $H_{\alpha_1,\alpha_2\alpha_3}$)
  with respect to the map in Lemma~\ref{twistisom} is
  \begin{align*}
   \frac{1-\chi(\gammaone)}{p^n}\eta_n(\gammatwo)\pi^{-1} \times 
   &(-p^{-2} f_1(\pi_n)d\log F_2^\phi(\pi_n)\wedge d\log F_3^\phi(\pi_n)\\
   &+p^{-1} f_2(\pi_n)d\log F_1(\pi_n)\wedge d\log F_3^\phi(\pi_n)\\
   &-f_3(\pi_n)d\log F_2(\pi_n)\wedge d\log F_1(\pi_n))\mod S_n.
  \end{align*}
  Taking into account that $p^{-n}\log\chi(\gammaone)=p^{-n}(\chi(\gammaone)-1)\mod p^n$, we obtain that 
  \begin{align*}
   -c\TR(H_{\alpha_1,\alpha_2,\alpha_3})=& -\Tr_{F\slash\QQ_p}\Res_{\pi_n,T} 
   \pi^{-1} (p^{-2} f_1(\pi_n)d\log F_2^\phi(\pi_n)\wedge d\log F_3^\phi(\pi_n)\\
   &-p^{-1} f_2(\pi_n)d\log F_1(\pi_n)\wedge d\log F_3^\phi(\pi_n)\\
   &+f_3(\pi_n)d\log F_2(\pi_n)\wedge d\log F_1(\pi_n)),
  \end{align*}
  which finishes the proof.



\vspace{3ex}

 Saral Livia Zerbes
 
 Department of Mathematics
 
 Imperial College London
 
 London SW7 2AZ
 
 United Kingdom
 
 {\it email:} s.zerbes@imperial.ac.uk


\begin{thebibliography}{1}

\bibitem{andreatta}
F.~Andreatta, \emph{Generalized rings of norms and generalized
  $(\phi,{\Gamma})$-modules}, A.. Sci. \'Ecole Norm. Sup. (4) \textbf{39}
  (2006), 599--647.

\bibitem{benois1}
D.~Benois, \emph{On {I}wasawa theory of crystalline representations}, Duke
  Math. Journal \textbf{104} (2000), 211--267.

\bibitem{brueckner}
H.~Br\"uckner, \emph{Explizites {R}eziprozit\"atsgesetz und {A}nwedungen},
  Vorlesungen aus dem Fachbereich Mathematik der Universit\"at Essen, 1979.

\bibitem{cherbonniercolmez2}
F.~Cherbonnier and P.~Colmez, \emph{Th\'eorie d'{I}wasawa des repr\'esentations
  {$p$}-adiques d'un corps local}, J. Amer. Math. Soc. \textbf{12} (1999),
  no.~1, 241--268.

\bibitem{herr1}
L.~Herr, \emph{Sur la cohomologie galoisienne des corps $p$-adiques}, Bull.
  Soc. Math. France \textbf{126} (1998), 563--600.

\bibitem{herr2}
L.~Herr, \emph{Une approche nouvelle de la dualit\'e locale de {T}ate}, Math.
  Ann. \textbf{320} (2001), 307--337.

\bibitem{scholl}
A.J. Scholl, \emph{Higher {F}ields of {N}orms and $(\phi,\gamma)$-{M}odules},
  Document Math. \textbf{Extra Volume: J. H. Coates' Sixtieth Birthday} (2006),
  711--727.

\bibitem{vostokov}
S.~V. Vostokov, \emph{Hilbert pairing on a comlete multidimensional local
  field}, Trudy Sankt-Peterb. Mat. Obschetsva (1995), no.~2, 111--148.

\bibitem{zerbesexp}
S.L. Zerbes, \emph{Higher exponential maps and explicit reciprocity laws I}, preprint.


\end{thebibliography}
\end{document}